\documentclass[11pt,twoside,reqno]{amsart}
\RequirePackage{newlfont}
\usepackage{amscd}
\usepackage{amsmath}
\usepackage{amsfonts}
\usepackage{amssymb}
\usepackage{enumerate}
\usepackage{graphicx}
\usepackage{latexsym}
\usepackage{color}
\usepackage{bbm}
\usepackage{hyperref}
\hypersetup{
    colorlinks=true,
    linkcolor=blue,
    citecolor=blue,
    urlcolor=blue,
    pdfborder={0 0 0}
}
\newtheorem{theorem}{Theorem}[section]
\newtheorem{lemma}[theorem]{Lemma}
\newtheorem{proposition}[theorem]{Proposition}
\newtheorem{corollary}[theorem]{Corollary}
\theoremstyle{definition}
\newtheorem{definition}[theorem]{Definition}

\newtheorem{remark}[theorem]{Remark}

\theoremstyle{remark}

\numberwithin{equation}{section}

\DeclareMathOperator*{\esssup}{esssup}

\begin{document}
\bibliographystyle{plainnat}

\setcounter{page}{1}

\title[Double-Phase Poisson Equations with Variable Growth]{Existence and Uniqueness for Double-Phase Poisson Equations with Variable Growth}

\author[M. A. Khamsi \& O. M\'{e}ndez]{Mohamed  A. Khamsi, Osvaldo M\'{e}ndez}

\address{Mohamed A. Khamsi\\ Department of Applied Mathematics and Sciences, Khalifa University, Abu Dhabi, UAE}
\email{mohamed.khamsi@ku.ac.ae}
\address{Osvaldo M\'{e}ndez\\Department of Mathematical Sciences, The University of Texas at El Paso, El Paso, TX 79968, USA}
\email{osmendez@utep.edu}

\subjclass[2020]{35J60, 35J92, 46E30}
\keywords{Double-phase operators; Modular function spaces; Nonlinear elliptic equations; Poisson equations; Uniform convexity; Variable exponent Sobolev spaces; Variational methods.}

\maketitle
\begin{abstract}
We study a class of nonlinear elliptic problems driven by a double-phase operator with variable exponents, arising in the modeling of heterogeneous materials undergoing phase transitions. The associated Poisson problem features a combination of two distinct growth conditions, modulated by a measurable weight function \( \mu \), leading to spatially varying ellipticity. Working within the framework of modular function spaces, we establish the uniform convexity of the modular associated with the gradient term. This structural property enables a purely variational treatment of the problem. As a consequence, we prove existence and uniqueness of weak solutions under natural and minimal assumptions on the variable exponents and the weight.
\end{abstract}

\medskip
\section{Introduction}

Partial differential equations with nonstandard growth have emerged as a central theme in modern nonlinear analysis, driven by applications to heterogeneous materials, electrorheological fluids, and image processing. Among these, the class of \emph{double-phase problems} has proven to be particularly rich in structure and challenging from both analytical and variational standpoints. Introduced by Zhikov \cite{Zhikov1986} to model materials with spatially varying hardening properties, the associated energy functionals take the form
\[
{\mathcal I}(u) = \int_{\Omega} \left( \frac{1}{p(x)} |\nabla u|^{p(x)} + \mu(x) \frac{1}{q(x)} |\nabla u|^{q(x)} \right) dx,
\]
where \( \mu : \Omega \to [0, \infty) \) is a measurable weight function controlling the transition between two elliptic phases. In regions where \( \mu(x) = 0 \), the energy exhibits \( p(x) \)-growth, while in regions where \( \mu(x) > 0 \), the dominant behavior is governed by \( q(x) \). This switching behavior—determined pointwise by \( \mu(x) \)—is what gives rise to the term “double-phase.”

Since Zhikov’s pioneering work, significant advances have been made in the understanding of these problems, especially in the regularity theory \cite{Baroni-Colombo-Mingione, Beck-Mingione, Colombo-Mingione-Regularity}, and their extensions to problems with variable exponents \cite{Bahrouni-Radulescu-Repovs, DeFilippis-Palatucci, Ragusa-Tachikawa}. We refer the reader to the survey \cite{IC} for a detailed account of the various lines of investigation emanating from the consideration of the double-phase problems. These developments highlight the analytical subtleties caused by nonuniform ellipticity, loss of homogeneity, and the failure of standard growth conditions.

In this work, we investigate the double-phase Poisson problem with variable exponents:
\begin{equation}\label{DPPP}
\left\{
\begin{array}{rrr}
-\operatorname{div} \left( |\nabla u|^{p(x)-2} \nabla u + \mu(x) |\nabla u|^{q(x)-2} \nabla u \right) = f &\quad \text{in }& \Omega, \\
u = \varphi &\quad \text{on }& \partial\, \Omega,
\end{array}
\right.
\end{equation}
as well as its natural generalization to multi-phase structures,
\begin{equation}\ref{MPPP}
\left\{
\begin{array}{rrr}
-\operatorname{div} \left( |\nabla u|^{p(x)-2} \nabla u + \sum\limits_{j=1}^{k} \mu_j(x) |\nabla u|^{q_j(x)-2} \nabla u \right) = f &\quad \text{in }& \Omega, \\
u = \varphi &\quad \text{on }& \partial\, \Omega.
\end{array}
\right.
\end{equation}

The corresponding differential operator associated to the boundary value problem (\ref{DPPP}),
\begin{equation}\label{operatord}
\Delta^{p}_{q,\mu}(u) := \operatorname{div} \left( |\nabla u|^{p(x)-2} \nabla u + \mu(x) |\nabla u|^{q(x)-2} \nabla u \right),
\end{equation}
is highly nonlinear and exhibits spatially varying ellipticity. The mathematical analysis of such operators faces several technical difficulties including the Lavrentiev phenomenon, lack of uniform ellipticity, and the failure of standard convexity conditions.

While much of the literature approaches these problems using  techniques based on the normed-space structure of variable exponent Sobolev spaces \cite{DHHR, KR}, our approach is distinctly modular. We adopt the framework of modular function spaces, whose study was initiated in \cite{Amine-Kozlowski-book} and show that the associated modular \( \varrho_H \) admits uniform convexity under minimal assumptions. This geometric property enables us to develop a variational theory free from restrictions on the norm uniform convexity, and in particular allows for unbounded exponents.

Our main result establishes existence and uniqueness of weak solutions to the double-phase Poisson problem under natural assumptions:

\begin{theorem}\label{core}
Let \( \Omega \subset \mathbb{R}^n \) be a bounded Lipschitz domain, and let \( p, q \in C(\overline{\Omega}) \) with \( 1 < p_-, q_- \) and \( p_+, q_+ < \infty \), and let  \( \mu : \Omega \to [0, \infty) \), $\mu\in L^1_{\mathrm{loc}}(\Omega)$, that is $\mu \in L^{1}(\Omega^{\prime})$ for every compact $\Omega^{\prime}\subset \Omega$. Then, for every \( \varphi \in W^{1,H}(\Omega) \) and every \( f \in \left(W_0^{1,H}(\Omega)\right)^* \), there exists a unique weak solution \( w \in W^{1,H}(\Omega) \) to the double-phase Poisson problem (\ref{DPPP}).
\end{theorem}

The paper is organized as follows. In Section~\ref{wh}, we introduce the modular Sobolev space \( W^{1,H}(\Omega) \) and develop its basic structure. Section~\ref{wh-UC} establishes the modular uniform convexity of the associated energy functionals. In Section~\ref{sol}, we formulate the variational problem and prove Theorem~\ref{core} using a direct method within the modular framework.

\section{The  modular Sobolev space \( W^{1,H}(\Omega) \)}\label{wh}
This section is devoted to the construction of the space \( W^{1,H}(\Omega) \), along with its modular structure, norm properties, and embedding results, which form the analytical backbone of our study.

\smallskip

Let \( \Omega \subset \mathbb{R}^n \), \( n \geq 2 \), be a bounded domain with Lipschitz boundary \( \partial\, \Omega \). Let \( p, q , \mu : \Omega \to \mathbb{R} \) be Borel measurable functions such that \( 1 < p(x), q(x) < \infty \) and $\mu(x) \geq 0$ for almost every \( x \in \Omega \). 
Define the function
$$H(x,t) := \frac{1}{p(x)} t^{p(x)} + \frac{\mu(x)}{q(x)} t^{q(x)}, \qquad (x,t) \in \Omega \times [0,\infty).$$
It can be easily verified that $H$ is a Musielak-Orlicz function for which
\(\inf 
\limits_{\Omega}H(x,1)\geq \frac{1}{p_+}>0\); in addition, $H$
satisfies a generalized growth condition of Musielak-Orlicz type; specifically, for \( \alpha(x)=\frac{1}{p_+}+\frac{\mu(x)}{q_+}, \beta (x)=\frac{1}{p_-}+\frac{\mu(x)}{q_-} \), one has
\[
\alpha (x)\min\{|\xi|^{p(x)},|\xi|^{q(x)}\} \leq H(x, \xi) \leq \beta(x) \max\{|\xi|^{p(x)},|\xi|^{q(x)}\}
\]
for a.e.\ \( x \in \Omega \) and all \( \xi \in \mathbb{R}^n \).
The associated functional  given by
\begin{equation*}
\varrho_H(u) := \int_\Omega \left( \frac{1}{p(x)} |u(x)|^{p(x)} + \frac{\mu(x)}{q(x)} |u(x)|^{q(x)} \right) dx
\end{equation*}
defines a convex modular on ${\mathcal M}$, the set of extended-real-valued Borel-measurable functions defined on $\Omega$, that is
\begin{enumerate}\label{modularproperties}
\item[(1)] $\varrho_H(u) = 0$ if and only if $u = 0$,
\item[(2)] $\varrho_H(\alpha u) = \varrho(x)$, if $|\alpha| =1$,
\item[(3)] $\varrho_H(\alpha u + (1-\alpha) v )\leq \alpha\varrho_H(u) + (1-\alpha)\varrho_H(v)$, for any $\alpha \in [0,1]$
and any $u, v \in {\mathcal M}$.
\end{enumerate}
The variable exponent Lebesgue space $L^H(\Omega)$ is defined as \cite{DHHR,KR}
\begin{equation*}
L^H(\Omega)=\left\{u\in {\mathcal M}:\varrho_H(\lambda u)<\infty\,\,\text{for some}\,\,\lambda>0\right\},
\end{equation*}
which  is a Banach space when endowed with the Luxemburg norm
$$\|u\|_H :=\inf\left \{\lambda>0:\varrho_H\left(\frac{u}{\lambda}\right) \leq 1 \right\}.$$
Analogously, on the set ${\mathcal V}\subset {\mathcal M}$ consisting of those functions whose distributional derivatives belong to ${\mathcal M}$ one has the convex modular $\varrho_{1,H}:{\mathcal V}\rightarrow [0,\infty]$ defined by

\begin{align}\label{defrhoHsob}
\varrho_{1,H}(u):&= \varrho_H(u)+\varrho_H\left(|\nabla u|\right) \\
\nonumber &= \int_\Omega \left( \frac{1}{p(x)} |u(x)|^{p(x)} + \frac{\mu(x)}{q(x)} |u(x)|^{q(x)} \right) dx \\
\nonumber &\quad \quad \quad \quad + \int_\Omega \left( \frac{1}{p(x)} |\nabla u(x)|^{p(x)} + \frac{\mu(x)}{q(x)} |\nabla u(x)|^{q(x)} \right) dx,
\end{align}
where $|\nabla u|$ stands for the Euclidean norm (in $\mathbb{R}^n$) of the gradient of $u$.  The space $W^{1,H}(\Omega)$ is defined as the class of real-valued, Borel measurable functions $u \in {\mathcal V}$ , for which there exists $\lambda>0$ such that $\varrho_{1,H}(\lambda u)<\infty.$  The Luxemburg norm $\|\cdot\|_{1,H}$ associated to the modular $\varrho_{1,H}$ is given by
$$\|u\|_{1,H}=\inf\left \{\lambda>0:\varrho_{1,H}\left(\frac{u}{\lambda}\right) \leq 1 \right\},$$
with respect to which, $W^{1,H}(\Omega)$ is a Banach space  (see \cite{CGHW}).\\

 It is straightforward to verify that if the generalized function $H$ is replaced with 
\begin{equation*}
\overline{H}(x,t)= |u(x)|^{p(x)} + \mu(x) |u(x)|^{q(x)}, 
\end{equation*}
then $L^H(\Omega)$ and  $L^{\overline{H}}(\Omega)$ ($W^{1,H}(\Omega)$ and $W^{1,\overline{H}}(\Omega)$) are identical sets and the norms \( \|\cdot\|_{H} \) and \( \|\cdot\|_{\overline{H}} \) (and similarly \( \|\cdot\|_{1,H} \) and \( \|\cdot\|_{1,\overline{H}} \)) are equivalent.   Moreover, if \( p_+ := \esssup\limits_{x \in \Omega}\ p(x) < \infty \) and \( q_+ := \esssup\limits_{x \in \Omega}\ q(x) < \infty \), then \( \varrho_H \)-convergence is equivalent to \( \varrho_{\overline{H}} \)-convergence, that is to say, for any sequence $(u_j)\subset L^H(\Omega)$, 
$$\varrho_{H}(u_j)\rightarrow 0 \,\,\text{as}\,\, j\rightarrow\infty\iff \varrho_{\overline{H}}(u_j)\rightarrow 0\,\,\text{as}\,\,j\rightarrow \infty.$$
Indeed, the equivalence of the norms follows from the obvious inequality $\|u\|_H\leq \|u\|_{\overline{H}}$  and the bound
$\frac{1}{e^{\frac{1}{e}}}\leq\frac{1}{x^{\frac{1}{x}}}\leq 1$, valid for $x\geq 1$, which implies
\[\varrho_H\left(\frac{u}{\lambda}\right)\leq 1\Rightarrow \varrho_{\overline {H}}\left(\frac{u}{\lambda e^{\frac{1}{e}}}\right)\leq 1,\]
whence $e^{\frac{1}{e}}\|u\|_H\geq \|u\|_{\overline{H}}$. \\

If, in addition, \( p, q \in C(\overline{\Omega}) \), then they are automatically bounded on \( \Omega \), and hence bounded away from infinity.  Throughout, we set
\[
m := \min\{p_-, q_-\}, \qquad M := \max\{p_+, q_+\}.
\]

\medskip
\begin{proposition}\label{normtomodular}\cite{DHHR,KR,ML}
Let \( \Omega \subseteq \mathbb{R}^n \) be an open set and suppose \(1\leq  p, q \in C(\overline{\Omega}) \) Then, for all \( u \in L^H(\Omega) \), one has
\[
\min\left\{ \left( \varrho_H(u) \right)^{\frac{1}{m}}, \left( \varrho_H(u) \right)^{\frac{1}{M}} \right\}
\leq \|u\|_H
\leq
\max\left\{ \left( \varrho_H(u) \right)^{\frac{1}{m}}, \left( \varrho_H(u) \right)^{\frac{1}{M}} \right\}.
\]
\end{proposition}
\begin{proof}  Without loss of generality, we may assume that \( u \neq 0 \).
It is clear from the definition of the Luxemburg norm that 
\[
\varrho_H\left(\frac{u}{\|u\|_H}\right)\leq 1;
\] 
Since $M = \max\{p_+,q_+\}<\infty$, it must hold $\varrho_H\left(\frac{u}{\|u\|_H}\right)= 1$; to see this notice that for all $\theta:0<\theta<\|u\|_H$ it must hold
\begin{equation*}
1<\varrho_H\left(\frac{u}{\theta}\right)=	\varrho_H\left(\frac{u}{\|u\|_H}\frac{\|u\|_H}{\theta}\right)\leq \left(\frac{\|u\|_H}{\theta}\right)^{M}\varrho_H\left(\frac{u}{\|u\|_H}\right);
\end{equation*}
taking the limit as $\theta \rightarrow 1$ in the preceding inequality it is easy to see that $\varrho_H\left(\frac{u}{\|u\|_H}\right)\geq 1$.\\
 It is easy to check that if $\|u\|_H>1$ it holds
$$\frac{1}{\|u\|_H^{M}}\ \varrho_H(u)\leq 	\varrho_H\left(\frac{u}{\|u\|_H}\right) =1=\varrho_H\left(\frac{u}{\|u\|_H}\right) \leq \frac{1}{\|u\|_H^{m}}\ \varrho_H(u),$$
whereas if $\|u\|_H\geq 1$ the corresponding inequalities are
$$\frac{1}{\|u\|_H^{m}}\ \varrho_H(u)\leq 	\varrho_H\left(\frac{u}{\|u\|_H}\right) =1=\varrho_H\left(\frac{u}{\|u\|_H}\right) \leq \frac{1}{\|u\|_H^{M}}\ \varrho_H(u).$$
The proof of the proposition follows from the preceding estimates.
\end{proof}

\medskip
Under the assumption $\mu\in L_{\mathrm{loc}}^{1}(\Omega)$  it is immediate that $C^{\infty}_0(\Omega)\subset W^{1,H}(\Omega)$ and the space \( W^{1,H}_0(\Omega) \) is defined as the \( \|\cdot\|_{1,H} \)-closure of \( C^\infty_0(\Omega) \) in \( W^{1,H}(\Omega) \). As usual, its dual is denoted by \( \left(W^{1,H}_0(\Omega)\right)^* \).

\medskip
\begin{proposition}\label{equivalentnorm} \cite{LaMen}
Assume that \( \mu \in L_{\mathrm{loc}}^{1}(\Omega) \). Define the Matuszewska index \( m(x) \) of the function \( H(x, \cdot) \) by
\[
m(x) := \inf_{t > 1}\  \frac{\ln M(x,t)}{\ln t}, \quad \text{where} \quad 
M(x,t) :=\ \limsup_{u \to \infty}\ \frac{H(x, tu)}{H(x, u)}.
\]
Then \( m(x) = \max\{p(x), q(x)\} \), and if \( p, q \in C(\overline{\Omega}) \) with \( \min\{p_-, q_-\} > 1 \), then \( m \in C(\overline{\Omega}) \). Moreover, the following properties hold:
\begin{enumerate}
\item The embedding \( W^{1,H}_0(\Omega) \hookrightarrow L^H(\Omega) \) is compact.
\item There exists a constant \( C = C(p, q, \mu, \Omega) \) such that
\begin{equation}\label{poincare}
\|u\|_{L^H} \leq C \Big\| |\nabla u| \Big\|_H, \quad \text{for all } u \in W^{1,H}_0(\Omega).
\end{equation}
This inequality is known as the Poincar\'{e} inequality.
\item The functional
\[
u \mapsto \Big\| |\nabla u| \Big\|_H
\]
defines a norm on \( W^{1,H}_0(\Omega) \) that is equivalent to the Luxemburg norm \( \|\cdot\|_{1,H} \).
\end{enumerate}
\end{proposition}

\medskip
Using the previous results, we deduce the following important corollary:

\begin{corollary}\label{fboundedbelow}
Let \( f \in \left(W^{1,H}_0(\Omega)\right)^* \) with dual norm \( \|f\|_{(W^{1,H}_0)^*} \). Then, for all \( u \in W^{1,H}_0(\Omega) \) such that \( \big\| |\nabla u| \big\|_H \geq 1 \), we have
\[
\big| \langle f, u \rangle \big| \leq \|f\|_{(W^{1,H}_0)^*} \, \varrho_H(u)^{\frac{1}{m}}.
\]
\end{corollary}

\medskip
The following lemma, which plays a crucial role in the proof of the main result for the double-phase Poisson problem, is a direct application of the previous results.

\begin{lemma}\label{boundedbelow} 
Let \( \Omega \subseteq \mathbb{R}^n \) be an open set, and let \( p, q \in C(\overline{\Omega}) \). Assume \( m = \min\{p_-, q_-\} > 1 \) and \( \mu\in L_{\mathrm{loc}}^{1}(\Omega) \). Suppose \( f \in \left(W^{1,H}_0(\Omega)\right)^* \) and \( \varphi \in W^{1,H}(\Omega) \). Then the functional \( \mathcal{I} : W^{1,H}_0(\Omega) \to [0, \infty) \) defined by
\[
\mathcal{I}(u) := \varrho_H(\varphi - u) - \langle f, u \rangle
\]
is bounded below. 
\end{lemma}

\begin{proof}
By Proposition~\ref{equivalentnorm}, we have
\[
|\langle f, u\rangle| \leq \|f\|_{\left(W^{1,H}_0(\Omega)\right)^*} \ \Big\| |\nabla u| \Big\|_H,
\]
which implies
\[
|\langle f, u\rangle| \leq \|f\|_{\left(W^{1,H}_0(\Omega)\right)^*} \left( \| |\nabla(u - \varphi)| \|_H + \| |\nabla \varphi| \|_H \right).
\]
According to Corollary~\ref{fboundedbelow}, we know that
\[
\| |\nabla(u - \varphi)| \|_H \leq \max\left\{ 1, \left( \varrho_H(u - \varphi) \right)^{\frac{1}{\min\{p_-, q_-\}}} \right\} = \max\left\{ 1, \left( \varrho_H(u - \varphi) \right)^{\frac{1}{m}} \right\}.
\]
Therefore, we obtain
\begin{align*}
\mathcal{I}(u) &= \varrho_H(\varphi - u) - \langle f, u \rangle \\
&\geq \varrho_H(\varphi - u) - \|f\|_{\left(W^{1,H}_0(\Omega)\right)^*} \Big( \| |\nabla(u - \varphi)| \|_H + \| |\nabla \varphi| \|_H \Big) \\
&\geq \varrho_H(\varphi - u) - \|f\|_{\left(W^{1,H}_0(\Omega)\right)^*} \left( \max\left\{ 1, \left( \varrho_H(u - \varphi) \right)^{\frac{1}{m}} \right\} + \| |\nabla \varphi| \|_H \right) \\
&\geq \varrho_H(\varphi - u) - \|f\|_{\left(W^{1,H}_0(\Omega)\right)^*} \left( 1 + \left( \varrho_H(u - \varphi) \right)^{\frac{1}{m}} + \| |\nabla \varphi| \|_H \right) \\
&= \varrho_H(\varphi - u) - \|f\|_{\left(W^{1,H}_0(\Omega)\right)^*} \left( \varrho_H(u - \varphi)^{\frac{1}{m}} \right) - \|f\|_{\left(W^{1,H}_0(\Omega)\right)^*} \left( 1 + \| |\nabla \varphi| \|_H \right).
\end{align*}
Since \( m > 1 \), we apply the inequality
\[
x - a x^{1/m} \geq -a \left( \frac{a}{m} \right)^{\frac{1}{m - 1}}, \quad \text{for all } x \geq 0,
\]
with \( x = \varrho_H(\varphi - u) \) and \( a = \|f\|_{\left(W^{1,H}_0(\Omega)\right)^*} \), to deduce
\[
\mathcal{I}(u) \geq -\|f\|_{\left(W^{1,H}_0(\Omega)\right)^*} \left( \frac{1}{m} \, \|f\|_{\left(W^{1,H}_0(\Omega)\right)^*} \right)^{\frac{1}{m - 1}} - \|f\|_{\left(W^{1,H}_0(\Omega)\right)^*} \left( 1 + \| |\nabla \varphi| \|_H \right).
\]
The right-hand side is finite, and thus \( \mathcal{I}(u) \) is bounded below. This completes the proof of Lemma~\ref{boundedbelow}.  
\end{proof}

\medskip
The modular structure of the Sobolev space \( W^{1,H}(\Omega) \) provides the functional setting in which we formulate and analyze the double-phase Poisson problem. In the next section, we investigate a fundamental geometric property of this space: the modular uniform convexity.

\section{Uniform Convexity of the Modular in \( W^{1,H}(\Omega) \)}\label{wh-UC}
Uniform convexity is a central feature in the analysis of nonlinear elliptic problems, as it ensures both strict convexity of energy functionals and strong convergence properties of minimizing sequences. In this section, we establish the uniform convexity of the modular \( \varrho_H \) associated with the gradient part of \( W^{1,H}(\Omega) \), under minimal assumptions on the exponent functions. This result will be a key ingredient in our variational approach to the double-phase Poisson problem.\\

The following definition is at the center stage of this work
\begin{definition}\label{defuc}
A convex semimodular $\varrho$ on a vector space $V$ (that is, a functional $\varrho$ on $V$ that satisfies $\varrho(0)=0$ and properties $(2)$and $(3)$ but not necessarily $(1)$ in the definition of a modular) is said to be uniformly convex (in short $UC$) if for every $\varepsilon>0$ there exists $\delta=\delta(\varepsilon): 0<\delta<1$ such that  for every $u\in V_\varrho$ and $v\in V_\varrho$:
$$\varrho\left(\frac{u-v}{2}\right) \geq \varepsilon\ \frac{\varrho(u)+\varrho(v)}{2} \;\; \text{implies}\;\;\;
\varrho\left(\frac{u+v}{2}\right) \leq (1-\delta)\ \frac{\varrho(u)+\varrho(v)}{2}.$$
\end{definition}

\smallskip
\begin{remark}\label{sum-uc}
It is worth noting that the sum of two uniformly convex semimodulars is also uniformly convex. This follows directly from Lemma~2.4.16 in \cite{DHHR}, which states precisely that if \( \varrho_1 \) and \( \varrho_2 \) are uniformly convex semimodulars on a vector space \( X \), then their sum \( \varrho = \varrho_1 + \varrho_2 \) is uniformly convex.
\end{remark}

\medskip

The inequalities presented in the following lemma form the foundation of the main result in this section. For a detailed proof, we refer the interested reader to \cite{KaMe}.

\medskip
\begin{lemma}\label{inequalities}
Let \( h \) be a measurable function on \( \Omega \), and suppose that \( 1 < h(x) < \infty \) almost everywhere.  Let \( u, v \in W^{1,h}(\Omega) \), and let \( |\cdot| \) denote the Euclidean norm in \( \mathbb{R}^n \).
\begin{enumerate}
\item[(i)] If \( 1 \leq h \leq 2 \), then the following inequality holds:
\begin{equation}\label{1<p<2-vector}
\frac{1}{h} \left| \frac{\nabla u + \nabla v}{2} \right|^h + \frac{(h - 1)}{2^{h + 1}} \frac{|\nabla u - \nabla v|^2}{(|\nabla u| + |\nabla v|)^{2 - h}} \leq \frac{1}{2h} \left( |\nabla u|^h + |\nabla v|^h \right),
\end{equation}
provided that \( |\nabla u| + |\nabla v| \neq 0 \).
\item[(ii)] If \( h \geq 2 \), then the following inequality holds:
\begin{equation}\label{p>2-vector}
\frac{1}{h} \left| \frac{\nabla u + \nabla v}{2} \right|^h + \frac{1}{h} \left| \frac{\nabla u - \nabla v}{2} \right|^h \leq \frac{1}{2h} \left( |\nabla u|^h + |\nabla v|^h \right).
\end{equation}
\end{enumerate}
\end{lemma}

\medskip

To establish the uniform convexity of the semimodular \( \varrho_H \), we begin by developing several technical lemmas that will serve as the foundation for the main result of our work. These lemmas isolate key estimates involving variable exponent integrals and provide the essential tools for handling the nonlinearities introduced by the presence of both \( p(x) \)- and \( q(x) \)-growth.\\

\noindent Let \( \Omega \) be a Lipschitz domain, and let \( p, q, \mu \) be Borel measurable functions on \( \Omega \) such that \( 1 < p(x), q(x) < \infty \) almost everywhere, and \( \mu \geq 0 \).  Notice that no further condition on $\mu$ is required for uniform convexity.  Before turning to the technical lemmas, we introduce several notations that will be used consistently throughout the analysis. For simplicity, the notation \( \varrho \) will be used in place of \( \varrho_H \), and for any subset \( X \subset \Omega \) and any function \( w \in W^{1,H}(\Omega) \), we denote
\[
\varrho^{X}(w) := \varrho_H\left(\mathbbm{1}_X |\nabla w|\right) = \int_X \left( \frac{1}{p}|\nabla w|^{p} + \frac{\mu}{q} |\nabla w|^{q} \right) dx.
\]
Define the subsets
\[
\left\{
\begin{aligned}
\Omega_{11} &:=  \{x \in \Omega : 1 < p(x) < 2 \text{ and } 1 < q(x) < 2\},\\
\Omega_{12} &:=  \{x \in \Omega : 1 < p(x) < 2 \text{ and } q(x) \geq 2\},\\
\Omega_{21} &:=  \{x \in \Omega : p(x) \geq 2 \text{ and } 1 < q(x) < 2\},\\
\Omega_{22} &:=  \{x \in \Omega : p(x) \geq 2, \, q(x) \geq 2\}.
\end{aligned}
\right.
\]
Notice that
\begin{equation}\label{decompositiondomain}
\Omega \setminus \Omega_{22} = \Omega_{12} \cup \Omega_{21} \cup \Omega_{11}.
\end{equation} 

\medskip
\noindent For fixed functions \( u, v \in W^{1,H}(\Omega) \) and a parameter \( \alpha > 0 \), specify the measurable subsets
\begin{align*}
G_{\alpha} &:= \left\{x \in \Omega : |\nabla(u - v)| \leq \frac{\alpha}{4} \left(|\nabla u| + |\nabla v|\right) \right\}, \\
E_{\alpha} &:= \left\{x \in \Omega : |\nabla(u - v)| > \frac{\alpha}{4} \left(|\nabla u| + |\nabla v|\right) \right\}.
\end{align*}

\noindent Set $A_\alpha = E_\alpha \cap \Omega_{21}$, $B_\alpha = E_\alpha \cap \Omega_{12}$ and $C_\alpha = E_\alpha \cap \Omega_{11}$.  Note that 
\begin{align*}
\Omega\setminus\left(\Omega_{22}\cup G_{\alpha}\right) &= 
\Big(\Omega\setminus\ \Omega_{22}\Big) \cap \Big(\Omega\setminus\ G_{\alpha}\Big)\\
& = \Big(\Omega_{12} \cup \Omega_{21} \cup \Omega_{11}\Big) \cap \Big(E_{\alpha}\Big)\\
&=A_\alpha \cup B_\alpha \cup C_\alpha.
\end{align*}

\begin{lemma}\label{lemma1}
Let the domain \( \Omega \) and the functions \( p, q, \mu \) be as described above. Let \( \varepsilon > 0 \). Suppose that
\[
\varrho^{\Omega_{22}}\left( \frac{u - v}{2} \right) > \frac{\varepsilon}{2} \cdot \frac{\varrho(u) + \varrho(v)}{2}
\]
for some \( u, v \in W^{1,H}(\Omega) \). Then the following inequality holds:
\[
\varrho\left( \frac{u + v}{2} \right) \leq \left(1 - \frac{\varepsilon}{2} \right) \  \frac{\varrho(u) + \varrho(v)}{2}.
\]
\end{lemma}
\begin{proof}
Owing to inequality~\((ii)\) of Lemma~\ref{inequalities}, applied separately to \( p \) and \( q \), one has
\begin{align*}
\varrho^{\Omega_{22}}\left(\frac{u+v}{2}\right)
&= \int_{\Omega_{22}} \left( \frac{1}{p} \left| \nabla\left( \frac{u+v}{2} \right) \right|^p + \frac{\mu}{q} \left| \nabla\left( \frac{u+v}{2} \right) \right|^q \right) dx \\
&\leq \int_{\Omega_{22}} \left( \frac{1}{2p} \left( |\nabla u|^p + |\nabla v|^p \right) - \frac{1}{p} \left| \frac{\nabla(u - v)}{2} \right|^p \right) dx \\
&\quad + \int_{\Omega_{22}} \left( \frac{\mu}{2q} \left( |\nabla u|^q + |\nabla v|^q \right) - \frac{\mu}{q} \left| \frac{\nabla(u - v)}{2} \right|^q \right) dx.
\end{align*}
Hence,
\[
\varrho^{\Omega_{22}}\left( \frac{u+v}{2} \right) \leq \frac{1}{2} \left( \varrho^{\Omega_{22}}(u) + \varrho^{\Omega_{22}}(v) \right) - \varrho^{\Omega_{22}}\left( \frac{u - v}{2} \right).
\]
Consequently, using the convexity of \( \varrho \), we get
\begin{align*}
\varrho\left( \frac{u + v}{2} \right)
&= \varrho^{\Omega_{22}}\left( \frac{u + v}{2} \right) + \varrho^{\Omega \setminus \Omega_{22}}\left( \frac{u + v}{2} \right) \\
&\leq \frac{1}{2} \left( \varrho^{\Omega_{22}}(u) + \varrho^{\Omega_{22}}(v) \right) - \varrho^{\Omega_{22}}\left( \frac{u - v}{2} \right)\\
&\quad\quad + \frac{1}{2} \left( \varrho^{\Omega \setminus \Omega_{22}}(u) + \varrho^{\Omega \setminus \Omega_{22}}(v) \right) \\
&= \frac{1}{2} \left( \varrho(u) + \varrho(v) \right) - \varrho^{\Omega_{22}}\left( \frac{u - v}{2} \right) \\
&\leq \frac{1}{2} \left( \varrho(u) + \varrho(v) \right) - \frac{\varepsilon}{2} \cdot \frac{\varrho(u) + \varrho(v)}{2} \\
&= \left( 1 - \frac{\varepsilon}{2} \right) \cdot \frac{\varrho(u) + \varrho(v)}{2}.
\end{align*}
This completes the proof of the Lemma \ref{lemma1}.
\end{proof}

\medskip
\begin{lemma}\label{lemma2}
Let the domain \( \Omega \) and the functions \( p, q, \mu \) be as described above. Assume \( m = \min\{p_-, q_-\} > 1 \), and let \( \varepsilon \in (0,1) \). Suppose that
\[
\varrho^{\Omega \setminus \Omega_{22}}\left( \frac{u - v}{2} \right) > \frac{\varepsilon}{2} \cdot \frac{\varrho(u) + \varrho(v)}{2}
\]
for some \( u, v \in W^{1,H}(\Omega) \). Then the following inequalities hold:
\begin{equation}\label{est3}
\varrho^{A_\varepsilon}\left( \frac{u + v}{2} \right) + (q_- - 1)\frac{\varepsilon}{8} \cdot \varrho^{A_\varepsilon}\left( \frac{u - v}{2} \right)
\leq \frac{1}{2} \left( \varrho^{A_\varepsilon}(u) + \varrho^{A_\varepsilon}(v) \right),
\end{equation}
\begin{equation}\label{est4}
\varrho^{B_\varepsilon}\left( \frac{u + v}{2} \right) + (p_- - 1)\frac{\varepsilon}{8} \cdot \varrho^{B_\varepsilon}\left( \frac{u - v}{2} \right)
\leq \frac{1}{2} \left( \varrho^{B_\varepsilon}(u) + \varrho^{B_\varepsilon}(v) \right),
\end{equation}
\begin{equation}\label{est5}
\varrho^{C_\varepsilon}\left( \frac{u + v}{2} \right) + (m - 1)\frac{\varepsilon}{8} \cdot \varrho^{C_\varepsilon} \left ( \frac{u - v}{2} \right)
\leq \frac{1}{2} \left( \varrho^{C_\varepsilon}(u) + \varrho^{C_\varepsilon}(v) \right).
\end{equation}
\end{lemma}
\begin{proof}  First, observe that
\begin{align*}
\Omega \setminus \Omega_{22} 
&= \left( \Omega \setminus \Omega_{22} \right) \cap \left( G_{\varepsilon} \cup (\Omega \setminus G_{\varepsilon}) \right) \\
&= \left( (\Omega \setminus \Omega_{22}) \cap G_{\varepsilon} \right) \cup \left( (\Omega \setminus \Omega_{22}) \cap (\Omega \setminus G_{\varepsilon}) \right) \\
&= \left( (\Omega \setminus \Omega_{22}) \cap G_{\varepsilon} \right) \cup \left( \Omega \setminus (\Omega_{22} \cup G_{\varepsilon}) \right).
\end{align*}
Since the sets \( (\Omega \setminus \Omega_{22}) \cap G_{\varepsilon} \) and \( (\Omega \setminus \Omega_{22}) \cap (\Omega \setminus G_{\varepsilon}) \) are disjoint, it follows that
\[
\varrho^{\Omega \setminus (\Omega_{22} \cup G_{\varepsilon})} \left( \frac{u - v}{2} \right)
= \varrho^{\Omega \setminus \Omega_{22}} \left( \frac{u - v}{2} \right)
- \varrho^{G_{\varepsilon} \cap (\Omega \setminus \Omega_{22})} \left( \frac{u - v}{2} \right),
\]
which yields the estimate
\begin{equation}\label{Est1}
\varrho^{\Omega \setminus (\Omega_{22} \cup G_{\varepsilon})} \left( \frac{u - v}{2} \right)
> \frac{\varepsilon}{2} \cdot \frac{\varrho(u) + \varrho(v)}{2}
- \varrho^{G_{\varepsilon} \cap (\Omega \setminus \Omega_{22})} \left( \frac{u - v}{2} \right).
\end{equation}
The last term is easy to estimate:
\begin{align*}
\varrho^{G_{\varepsilon} \cap (\Omega \setminus \Omega_{22})}\left( \frac{u - v}{2} \right)
&= \int_{G_{\varepsilon} \cap (\Omega \setminus \Omega_{22})} \left( \frac{1}{p} \left| \frac{\nabla(u - v)}{2} \right|^p + \frac{\mu}{q} \left| \frac{\nabla(u - v)}{2} \right|^q \right) dx \\
&\leq \int_{G_{\varepsilon} \cap (\Omega \setminus \Omega_{22})} \frac{1}{p} \left( \frac{\varepsilon}{8} (|\nabla u| + |\nabla v|) \right)^p dx \\
&\quad \quad \quad \quad + \int_{G_{\varepsilon} \cap (\Omega \setminus \Omega_{22})} \frac{\mu}{q} \left( \frac{\varepsilon}{8} (|\nabla u| + |\nabla v|) \right)^q dx.
\end{align*}
Using hte fact that \( \varepsilon < 1 \) it is readily seen that
\begin{align*}
\varrho^{G_{\varepsilon} \cap (\Omega \setminus \Omega_{22})}\left( \frac{u - v}{2} \right)
&\leq \frac{\varepsilon}{4} \int_{G_{\varepsilon} \cap (\Omega \setminus \Omega_{22})} \frac{1}{2p} \left( |\nabla u|^p + |\nabla v|^p \right) dx \\
&\quad \quad \quad \quad  \frac{\varepsilon}{4} \int_{G_{\varepsilon} \cap (\Omega \setminus \Omega_{22})} \frac{\mu}{2q} \left( |\nabla u|^q + |\nabla v|^q \right) dx \\
&\leq \frac{\varepsilon}{4} \cdot \frac{ \varrho^{G_{\varepsilon} \cap (\Omega \setminus \Omega_{22})}(u) + \varrho^{G_{\varepsilon} \cap (\Omega \setminus \Omega_{22})}(v) }{2} \\
&\leq \frac{\varepsilon}{4} \cdot \frac{ \varrho(u) + \varrho(v) }{2}.
\end{align*}
Substituting into \eqref{Est1} the following inequality is clear :
\begin{align}\label{Est2}
\varrho^{\Omega \setminus (\Omega_{22} \cup G_{\varepsilon})}\left( \frac{u - v}{2} \right)
\geq \frac{\varepsilon}{4} \cdot \frac{ \varrho(u) + \varrho(v) }{2}.
\end{align}
On the set \( A_\varepsilon \), we have \( p \geq 2 \) and \( 1 < q_- < 2 \), hence \( (q_- - 1)\frac{\varepsilon}{8} < 1 \).  It follows from Lemma~\ref{inequalities} that
\begin{align*}
\frac{1}{p} \left| \frac{\nabla(u+v)}{2} \right|^p + (q_- - 1)\frac{\varepsilon}{8} \cdot \frac{1}{p} \left| \frac{\nabla(u - v)}{2} \right|^p 
&\leq \frac{1}{p} \left| \frac{\nabla(u+v)}{2} \right|^p + \frac{1}{p} \left| \frac{\nabla(u - v)}{2} \right|^p \\
&\leq \frac{1}{2} \left( \frac{1}{p} |\nabla u|^p + \frac{1}{p} |\nabla v|^p \right),
\end{align*}
and
\begin{align*}
(q_- - 1)\frac{\varepsilon}{8} \cdot \frac{\mu}{q} \left| \frac{\nabla(u - v)}{2} \right|^q
&\leq \frac{\mu}{q} \left| \frac{\nabla(u - v)}{2} \right|^q \cdot \frac{q(q - 1)}{2} \left( \frac{\varepsilon}{4} \right)^{2 - q} \\
&\leq \left| \frac{\nabla(u - v)}{2} \right|^q \cdot \frac{\mu(q - 1)}{2} \left| \frac{\nabla(u - v)}{|\nabla u| + |\nabla v|} \right|^{2 - q} \\
&= \frac{\mu(q - 1)}{2^{q + 1}} \cdot \frac{|\nabla(u - v)|^2}{(|\nabla u| + |\nabla v|)^{2 - q}}.
\end{align*}
The application of inequality (i) of Lemma~\ref{inequalities} yields
\begin{align*}
\frac{\mu}{q} \left| \frac{\nabla(u+v)}{2} \right|^q + (q_- - 1)\frac{\varepsilon}{8} \cdot \frac{\mu}{q} \left| \frac{\nabla(u - v)}{2} \right|^q
&\leq \frac{\mu}{q} \left| \frac{\nabla(u+v)}{2} \right|^q \\
&\quad \quad \quad + \frac{\mu(q - 1)}{2^{q + 1}} \cdot \frac{|\nabla(u - v)|^2}{(|\nabla u| + |\nabla v|)^{2 - q}} \\
&\leq \frac{1}{2} \left( \frac{\mu}{q} |\nabla u|^q + \frac{\mu}{q} |\nabla v|^q \right).
\end{align*}
It follows by combining the previous inequalities and integrating on the set $A_\varepsilon$ that
\[
\varrho^{A_\varepsilon}\left( \frac{u + v}{2} \right) + (q_- - 1)\frac{\varepsilon}{8} \cdot \varrho^{A_\varepsilon}\left( \frac{u - v}{2} \right) 
\leq \frac{1}{2} \left( \varrho^{A_\varepsilon}(u) + \varrho^{A_\varepsilon}(v) \right).
\]
On the set \( B_\varepsilon \), we have \( q \geq 2 \) and \( 1 < p_- \leq p < 2 \). Using a similar argument to the one above and the positivity of \( \mu \), it is readily concluded that
\[
\varrho^{B_\varepsilon}\left( \frac{u + v}{2} \right) + (p_- - 1)\frac{\varepsilon}{8} \cdot \varrho^{B_\varepsilon}\left( \frac{u - v}{2} \right) 
\leq \frac{1}{2} \left( \varrho^{B_\varepsilon}(u) + \varrho^{B_\varepsilon}(v) \right).
\]
Finally, on the set \( C_\varepsilon \), it holds \( 1 < m \leq p_- \leq p < 2 \) and \( 1 < m \leq q_- \leq q < 2 \). On account of inequality~(i) of Lemma~\ref{inequalities}, one has
\begin{align*}
\frac{1}{p} \left| \frac{\nabla(u + v)}{2} \right|^p 
+ \frac{m - 1}{p} \cdot \frac{\varepsilon}{8} \left| \frac{\nabla(u - v)}{2} \right|^p 
&\leq \frac{1}{p} \left| \frac{\nabla(u + v)}{2} \right|^p 
+ \frac{p_- - 1}{p} \cdot \frac{\varepsilon}{8} \left| \frac{\nabla(u - v)}{2} \right|^p \\
&\leq \frac{1}{p} \left| \frac{\nabla(u + v)}{2} \right|^p 
+ \frac{p - 1}{2} \left( \frac{\varepsilon}{4} \right)^{2 - p} \left| \frac{\nabla(u - v)}{2} \right|^p \\
&\leq \frac{1}{p} \left| \frac{\nabla(u + v)}{2} \right|^p \\
&\quad\quad \quad + \frac{p - 1}{2} \left( \frac{|\nabla(u - v)|}{|\nabla u| + |\nabla v|} \right)^{2 - p} \left| \frac{\nabla(u - v)}{2} \right|^p \\
&\leq \frac{1}{2} \left( \frac{1}{p} |\nabla u|^p + \frac{1}{p} |\nabla v|^p \right).
\end{align*}
Since \( \mu \geq 0 \), the next inequality follows analogously, namely
\[
\frac{\mu}{q} \left| \frac{\nabla(u + v)}{2} \right|^q 
+ \frac{(m - 1)\varepsilon}{8} \cdot \frac{\mu}{q} \left| \frac{\nabla(u - v)}{2} \right|^q 
\leq \frac{1}{2} \left( \frac{\mu}{q} |\nabla u|^q + \frac{\mu}{q} |\nabla v|^q \right).
\]
Putting everything together and integrating over \( C_\varepsilon \), it follows 
\[
\varrho^{C_\varepsilon} \left( \frac{u + v}{2} \right) 
+ (m - 1)\frac{\varepsilon}{8} \cdot \varrho^{C_\varepsilon} \left( \frac{u - v}{2} \right) 
\leq \frac{1}{2} \left( \varrho^{C_\varepsilon}(u) + \varrho^{C_\varepsilon}(v) \right).
\]
This completes the proof of Lemma~\ref{lemma2}.
\end{proof}

\medskip
Using the previous lemmas, the stage is now set to state the following important result:

\begin{theorem}\label{theorem-uc}
Let the domain \( \Omega \) and the functions \( p, q, \mu \) be as described above. Assume \( m = \min\{p_-, q_-\} > 1 \). Let \( \varepsilon > 0 \) satisfy
\[
\varepsilon < \min\left\{1, \sqrt{\frac{32}{m - 1}} \right\}.
\]
Let \( u, v \in W^{1,H}(\Omega) \) be such that
\[
\varrho\left( \frac{u - v}{2} \right) > \varepsilon \cdot \frac{\varrho(u) + \varrho(v)}{2}.
\]
Then the following inequality holds:
\[
\varrho\left( \frac{u + v}{2} \right) \leq \left( 1 - \delta(\varepsilon) \right) \cdot \frac{\varrho(u) + \varrho(v)}{2},
\]
where
\[
\delta(\varepsilon) = \min\left\{ \frac{\varepsilon}{2}, \frac{(m - 1)\varepsilon^2}{32} \right\}.
\]
\end{theorem}
\begin{proof}
It is clear that under the stated assumption on \( u \) and \( v \), either
\begin{equation}\label{Ineq1}
\varrho^{\Omega_{22}}\left( \frac{u - v}{2} \right) > \frac{\varepsilon}{2} \cdot \frac{\varrho(u) + \varrho(v)}{2}
\end{equation}
or
\begin{equation}\label{Ineq2}
\varrho^{\Omega \setminus \Omega_{22}}\left( \frac{u - v}{2} \right) > \frac{\varepsilon}{2} \cdot \frac{\varrho(u) + \varrho(v)}{2}
\end{equation}
must hold.\\
If inequality \eqref{Ineq1} holds, it follows from Lemma~\ref{lemma1}
that
\begin{align*}
\varrho\left( \frac{u + v}{2} \right) 
&\leq \left( 1 - \frac{\varepsilon}{2} \right) \cdot \frac{\varrho(u) + \varrho(v)}{2} \\
&\leq \left( 1 - \delta(\varepsilon) \right) \cdot \frac{\varrho(u) + \varrho(v)}{2}.
\end{align*}
Otherwise inequality \eqref{Ineq2} holds and Lemma~\ref{lemma2} applies. Note that
\[
A_\varepsilon \cup B_\varepsilon \cup C_\varepsilon = \Omega \setminus (\Omega_{22} \cup G_\varepsilon) = (\Omega \setminus \Omega_{22}) \cap E_\varepsilon.
\]
Together with inequality~\eqref{Est2}, this implies
\[
\varrho^{A_\varepsilon \cup B_\varepsilon \cup C_\varepsilon}\left( \frac{u - v}{2} \right) 
\geq \frac{\varepsilon}{4} \cdot \frac{\varrho(u) + \varrho(v)}{2}.
\]
Owing to inequalities~\eqref{est3}, \eqref{est4}, and \eqref{est5} it follows that
\[
\varrho^{A_\varepsilon \cup B_\varepsilon \cup C_\varepsilon}\left( \frac{u + v}{2} \right) 
+ (m - 1)\frac{\varepsilon}{8} \cdot \varrho^{A_\varepsilon \cup B_\varepsilon \cup C_\varepsilon}\left( \frac{u - v}{2} \right) 
\leq \frac{1}{2} \left( \varrho^{A_\varepsilon \cup B_\varepsilon \cup C_\varepsilon}(u) + \varrho^{A_\varepsilon \cup B_\varepsilon \cup C_\varepsilon}(v) \right),
\]
which implies
\begin{align*}
\varrho^{A_\varepsilon \cup B_\varepsilon \cup C_\varepsilon}\left( \frac{u + v}{2} \right) 
\leq \frac{1}{2} \left( \varrho^{A_\varepsilon \cup B_\varepsilon \cup C_\varepsilon}(u) + \varrho^{A_\varepsilon \cup B_\varepsilon \cup C_\varepsilon}(v) \right) - (m - 1)\frac{\varepsilon^2}{32} \cdot \frac{\varrho(u) + \varrho(v)}{2}.
\end{align*}
On account of the convexity of the modular it also holds that
\[
\varrho^{\Omega \setminus (A_\varepsilon \cup B_\varepsilon \cup C_\varepsilon)}\left( \frac{u + v}{2} \right) 
\leq \frac{1}{2} \left( \varrho^{\Omega \setminus (A_\varepsilon \cup B_\varepsilon \cup C_\varepsilon)}(u) + \varrho^{\Omega \setminus (A_\varepsilon \cup B_\varepsilon \cup C_\varepsilon)}(v) \right).
\]
Combining both parts it follows that
\[
\varrho\left( \frac{u + v}{2} \right) 
\leq \frac{\varrho(u) + \varrho(v)}{2} - (m - 1)\frac{\varepsilon^2}{32} \cdot \frac{\varrho(u) + \varrho(v)}{2},
\]
which implies
\[
\varrho\left( \frac{u + v}{2} \right) 
\leq \left( 1 - \delta(\varepsilon) \right) \cdot \frac{\varrho(u) + \varrho(v)}{2}.
\]
This completes the proof of Theorem~\ref{theorem-uc}.
\end{proof}

\medskip
As an immediate consequence of Theorem~\ref{theorem-uc}, we arrive at the following fundamental result regarding the modular geometric properties of the semimodular $\varrho_H$.\\

\begin{theorem}\label{uctheorem}
Let \( \Omega \) be a Lipschitz domain, and let \( p, q, \mu \) be Borel measurable functions on \( \Omega \) such that \( 1 < m = \min\{p_-, q_-\} \) and \( \mu \geq 0 \). Then the semimodular defined by
\[
\varrho(u) = \varrho_H(|\nabla u|) = \int_{\Omega} \left( \frac{1}{p} |\nabla u|^p + \frac{\mu}{q} |\nabla u|^q \right) dx,
\]
for all \( u \in W^{1,H}(\Omega) \),
is uniformly convex.
\end{theorem}

\smallskip

The uniform convexity of the modular \( \varrho_H \) on \( L^H(\Omega) \) follows by an argument identical to the one used in Theorem~\ref{uctheorem}, applied to the zero-order term. Since both \( \varrho_H(u) \) and \( \varrho_H(|\nabla u|) \) are uniformly convex, it follows from Remark~\ref{sum-uc} that their sum is also uniformly convex. We thus obtain the following result:

\begin{corollary}\label{varrho1,H-uc}  Let \( \Omega \) be a Lipschitz domain, and let \( p, q, \mu \) be Borel measurable functions on \( \Omega \) such that \( 1 < m = \min\{p_-, q_-\} \) and \( \mu \geq 0 \).  Then the modular
\[
\varrho_{1,H}(u) = \varrho_H(u) + \varrho_H(|\nabla u|)
\]
is uniformly convex on \( W^{1,H}(\Omega) \).
\end{corollary}

\medskip
\noindent It is worth emphasizing that the uniform convexity of the modular \( \varrho_H \) remains valid even when the exponents \( p \) and \( q \) are unbounded on \( \Omega \), and \( \mu \) is merely Borel measurable and almost everywhere nonnegative, as long as \( m = \min\{p_-, q_-\} > 1 \). Theorem~\ref{uctheorem} applies mutatis mutandis to the modular \( \varrho_{1,H} \) in~\eqref{defrhoHsob}, showing that \( W^{1,H}(\Omega) \) is modularly uniformly convex even in cases where the Luxemburg norm fails to be uniformly convex—namely, when either \( p_+ = \infty \) or \( q_+ = \infty \). In particular, Theorem~\ref{uctheorem} remains valid without requiring continuity of the exponents or boundedness of the domain.

\medskip

As mentioned earlier, the double-phase Poisson problem can be extended to a more general multi-phase framework involving several distinct growth exponents and associated weights; see problem (\ref{MPPP}). In this setting, the modular functional is governed by the function
\[
\cal{H}(x,t) := \frac{1}{p(x)} t^{p(x)} + \sum_{j=1}^{k} \frac{\mu_j(x)}{q_j(x)} t^{q_j(x)},
\]
where \( p, q_j \in C(\overline{\Omega}) \), each \( \mu_j \geq 0 \) is Borel-measurable on \( \Omega \), and \( 1 < p(x), q_j(x)\) throughout \( \overline{\Omega} \). The associated modular is then given by
\[
\varrho_{\cal{H}}(u) := \int_{\Omega} \left( \frac{1}{p(x)} |u|^{p(x)} + \sum_{j=1}^{k} \frac{\mu_j(x)}{q_j(x)} |u|^{q_j(x)} \right) dx,
\]
as well as, the modular $\varrho_{1,\cal{H}}(u) = \varrho_{\cal{H}}(u) + \varrho_{\cal{H}}(|\nabla u|)$.  The associated modular vector space is denoted \( W^{1,\cal{H}}(\Omega) \).\\

\begin{theorem}\label{multi-uctheorem}
Let \( \Omega \) be a Lipschitz domain, and let \( p, q_1, \ldots, q_k : \Omega \to (1, \infty) \) and \( \mu_1, \ldots, \mu_k : \Omega \to [0, \infty) \) be Borel measurable functions. Then the  semimodular
\[
\varrho(u) = \varrho_{\mathcal{H}}(|\nabla u|) = \int_{\Omega} \left( \frac{1}{p(x)} |\nabla u|^{p(x)} + \sum_{j=1}^{k} \frac{\mu_j(x)}{q_j(x)} |\nabla u|^{q_j(x)} \right) dx
\]
is uniformly convex on \( W^{1,\mathcal{H}}(\Omega) \), provided \( m = \min\left\{ p_-, (q_1)_-, \ldots, (q_k)_- \right\} > 1 \). In particular, the modular  $\varrho_{1,\mathcal{H}}(u) = \varrho_{\mathcal{H}}(u) + \varrho_{\mathcal{H}}(|\nabla u|)$ is also uniformly convex on \( W^{1,\mathcal{H}}(\Omega) \).
\end{theorem}

\medskip
With the modular uniform convexity of \( W^{1,H}(\Omega) \) now fully established, we are ready to address the core analytic aspects of the double-phase Poisson problem. The next section is devoted to formulating the variational problem, analyzing the double-phase operator, and proving the existence and uniqueness of weak solutions.

\section{Solvability of the Double-Phase Poisson Problem}\label{sol}

In this section, we formulate the variational framework for the double-phase Poisson problem (\ref{DPPP}), analyze the associated Dirichlet integral, and establish the existence and uniqueness of weak solutions. The key ingredients in our approach are the uniform convexity of the modular established in Section~\ref{wh-UC} and the variational minimization method developed here.

\medskip
As established in Lemma~\ref{boundedbelow}, the functional \( \mathcal{I} : W^{1,H}_0(\Omega) \to \mathbb{R} \), defined by
\[
\mathcal{I}(u) = \varrho_H(\varphi - u) - \langle f, u \rangle,
\]
is bounded below. We denote by \( I_0 \) its infimum:
\[
-\infty < I_0 := \inf\, \Big\{\mathcal{I}(u); \, u \in W^{1,H}_0(\Omega)\Big\}\;  \leq \varrho_H(\varphi).
\]

\noindent The functional \( \mathcal{I}(u) = \varrho_H(\varphi - u) - \langle f, u \rangle \), while not representing the Dirichlet energy itself, plays an important auxiliary role in the analysis of the double-phase problem. It combines a modular deviation from a given function \( \varphi \) with a duality pairing involving the source term \( f \). The structure of \( \mathcal{I} \) is particularly suited for variational analysis, as it reflects the underlying convex geometry of the space. Its boundedness from below will be instrumental in the existence argument developed next.

\begin{theorem}\label{theoremexistence}  Let \( \Omega \subseteq \mathbb{R}^n \) be an open set, and let \( p, q \in C(\overline{\Omega}) \). Assume \( m = \min\{p_-, q_-\} > 1 \).  Suppose \( f \in \left(W^{1,H}_0(\Omega)\right)^* \) and \( \varphi \in W^{1,H}(\Omega) \). There exists a unique \( u_0 \in W^{1,H}_0(\Omega) \) that minimizes the functional \(\mathcal{I} \), that is,
\[
\mathcal{I}(u_0) = \min \Big\{ \mathcal{I}(u) \; ; \; u \in W^{1,H}_0(\Omega) \Big\} = I_0.
\]
\end{theorem}
\begin{proof}
Let \( (u_n)\subset W_0^{1,H}(\Omega) \) be a minimizing sequence of \( \mathcal{I} \), i.e.,  \( \lim\limits_{n \to \infty} \mathcal{I}(u_n) = I_0 \). Set \( \eta_n = \mathcal{I}(u_n) - I_0 \geq 0 \), for any \( n \in \mathbb{N} \). Clearly, \( \lim\limits_{n \to \infty} \eta_n = 0 \). First note that \( (u_n) \) is bounded in \( W^{1,H}_0(\Omega) \). Assume not. Without loss of generality, we assume that \( \|\nabla (u_n - \varphi)\|_{H} \rightarrow \infty \). Then \( \|\nabla (u_n - \varphi)\|_{H} > 1 \) for \( n \geq n_0 \), for some \( n_0 \in \mathbb{N} \), and thus, the inequality
\[
\Big\|\nabla (u_n - \varphi)\Big\|_{H}^{m} \leq \varrho_H(\varphi - u_n) \leq \mathcal{I}(u_n) + f(u_n),
\]
would imply
\begin{align*}
\Big\|\nabla (u_n - \varphi)\Big\|_{H}^{m}
&\leq \|f\|_{\left(W^{1,H}(\Omega)\right)^{\ast}} \|\nabla u_n\|_{H} + I_0 + \sup\limits_{k \in \mathbb{N}} \eta_k \\ \nonumber
&\leq \|f\|_{\left(W^{1,H}_0(\Omega)\right)^{\ast}} \|\nabla (u_n - \varphi)\|_{H} + \|f\|_{\left(W^{1,H}_0(\Omega)\right)^{\ast}} \|\nabla \varphi\|_{H} + I_0 + \sup\limits_{k \in \mathbb{N}} \eta_k,
\end{align*}
for all \( n \in \mathbb{N} \), which is certainly not possible since \( m > 1 \) and \( \lim\limits_{n \to \infty} \|\nabla (u_n - \varphi)\|_{H} = \infty \).\\
Consequently, 
\begin{equation}\label{boundedabove}
\varrho(u_n - \varphi) \leq \Big\| \nabla (u_n - \varphi) \Big\|^{\max\{p_+, q_+\}}_H \leq C,
\end{equation}
for a positive constant \( C \) independent of \( n \). Next, we prove that the minimizing sequence \( (u_n) \) is Cauchy in \( W_0^{1,H}(\Omega) \). Assume the contrary, then there exists \( \varepsilon > 0 \) and a subsequence of \( (v_j) \), say \( (w_k) \), for which
\[
\varrho\left( \frac{w_j - w_k}{2} \right) \geq \varepsilon,
\]
which implies 
\[
\varepsilon \leq \varrho\left( \frac{w_j - w_k}{2} \right) \leq \frac{\varrho(w_j - \varphi) + \varrho(w_k - \varphi)}{2}.
\]
Owing to (\ref{boundedabove}),
\[
C^{-1} \varepsilon \frac{\varrho(w_j - \varphi) + \varrho(w_k - \varphi)}{2} \leq \varrho\left( \frac{w_j - w_k}{2} \right).
\]
The uniform convexity established in Theorem~\ref{uctheorem} asserts the existence of \( \delta(\varepsilon, C) = \delta : 0 < \delta < 1 \) such that for all \( j, k \),
\[
\varrho\left( \frac{w_j + w_k}{2} - \varphi \right) \leq (1 - \delta) \frac{\varrho(w_j - \varphi) + \varrho(w_k - \varphi)}{2},
\]
which after adding \( -\langle f, \frac{w_j + w_k}{2} \rangle \) to both sides yields
\[
\varrho\left( \frac{w_j + w_k}{2} - \varphi \right) - \langle f, \frac{w_j + w_k}{2} \rangle \leq (1 - \delta) \frac{\varrho(w_j - \varphi) + \varrho(w_k - \varphi)}{2} - \langle f, \frac{w_j + w_k}{2} \rangle,
\]
or
\begin{align*}
\mathcal{I}\left( \frac{w_j + w_k}{2} \right) &\leq \frac{1}{2} \left( \mathcal{I}(w_j) + \mathcal{I}(w_k) \right) - \delta \frac{\varrho(w_j - \varphi) + \varrho(w_k - \varphi)}{2} \\
&\leq \frac{1}{2} \left( \mathcal{I}(w_j) + \mathcal{I}(w_k) \right) - \delta \varepsilon.
\end{align*}
This clearly leads to \( I_0 \leq I_0 - \delta \varepsilon \), which is a contradiction. Hence, the minimizing sequence \( (u_j) \) is \( \varrho_H \)-Cauchy. Owing to inequalities in Proposition~\ref{normtomodular}, it follows
\[
\|\nabla (u_j - u_k)\|_H \rightarrow 0 \quad \text{as } j, k \rightarrow \infty,
\]
and on account of Poincar\'{e}'s inequality (\ref{poincare}) one has \( \|u_j - u_k\|_H \rightarrow 0 \) as \( j, k \rightarrow \infty \). Since \( W^{1,H}_0(\Omega) \) is a Banach space, one deduces the existence of \( u \in W^{1,H}_0(\Omega) \) such that \( \|u_j - u\|_{1,H} \rightarrow 0 \) as \( j \rightarrow \infty \). Due to the equivalence between norm and modular convergence, one has \( \varrho_{1,H}(u_j - u) \rightarrow 0 \) as \( j \rightarrow \infty \). In particular, \( \varrho_H(u_j - u) \rightarrow 0 \), and by \cite{Musielak} there exists a subsequence \( (u_{j_k}) \) such that \( \nabla (u_{j_k} - u)(x) \rightarrow 0 \) a.e. in \( \Omega \). Fatou's lemma leads to
\[
\mathcal{I}(u) \leq \liminf\limits_{k \to \infty}\, \mathcal{I}(u_{j_k}) \leq I_0,
\]
i.e., \( u \) is a minimizer of \( \mathcal{I} \). The arbitrariness of the selected minimizing sequence yields the uniqueness of the minimizer.
\end{proof}

\medskip
\begin{remark}
It is well known that \( W_0^{1,H}(\Omega) \) is strictly contained in \( U_0^{1,H}(\Omega) \), which is defined as the norm closure of the set of functions in \( W^{1,H}(\Omega) \) with compact support. The uniqueness result established above shows, in particular, that the minimizer of \( \mathcal{I} \) over \( U_0^{1,H}(\Omega) \) actually belongs to \( W_0^{1,H}(\Omega) \).
\end{remark}

\medskip
We now turn to the variational properties of the functional \( \mathcal{I} \), beginning with its differentiability. The following technical lemma will play a key role in our analysis.

\medskip

\begin{lemma}\label{frechet-derivative}
Let \( \Omega \subseteq \mathbb{R}^n \) be a bounded open set, and let \( p, q \in C(\overline{\Omega}) \). Assume \( m = \min\{p_-, q_-\} > 1 \) and \( \mu \in L^1_{\mathrm{loc}}(\Omega) \). Consider the operator \( A : W^{1,\mathcal{H}}(\Omega) \to \left( W^{1,\mathcal{H}}(\Omega) \right)^* \) defined by
\[
\langle A(u), v \rangle_{\mathcal{H}} = \int_\Omega \left( |\nabla u|^{p(x)-2} \nabla u \cdot \nabla v + \mu(x) |\nabla u|^{q(x)-2} \nabla u \cdot \nabla v \right) dx,
\]
for all \( u, v \in W^{1,\mathcal{H}}(\Omega) \), and the associated energy functional \( {\mathcal I} : W^{1,H}(\Omega) \to \mathbb{R} \) defined by
\[
{\mathcal I}(u) = \int_\Omega \left( \frac{|\nabla u|^{p(x)}}{p(x)} + \mu(x) \frac{|\nabla u|^{q(x)}}{q(x)} \right) dx.
\]
Then \( {\mathcal I} \) is well-defined, Gâteaux differentiable, and satisfies \( {\mathcal I}'(u) = A(u) \).
\end{lemma}
\begin{proof}
In this proof, we primarily focus on the \( q \)-term due to the presence of the weight function \( \mu \). The argument for the \( p \)-term is analogous and can be obtained by setting \( \mu(x) = \mathbbm{1}_\Omega \), which belongs to \( L^1_{\mathrm{loc}}(\Omega) \) since \( \Omega \) is bounded.\\
Let \( u, h \in W^{1,H}(\Omega) \). We claim that \( \mu\, |\nabla u|^{q-2} (\nabla u \cdot \nabla h) \in L^1(\Omega) \). Indeed, let \( q' = \frac{q}{q - 1} \) be the conjugate exponent. Applying Young’s inequality, we find
\begin{align*}
\left| \mu\, |\nabla u|^{q-2} (\nabla u \cdot \nabla h) \right| &\leq \mu\, |\nabla u|^{q-2} |\nabla u|\, |\nabla h| \\
&= \mu\, |\nabla u|^{q-1} |\nabla h| \\
&\leq \frac{\mu}{q'}\, \left[ |\nabla u|^{q-1} \right]^{q'} + \frac{\mu}{q}\, |\nabla h|^q \\
&= \mu\, \frac{q - 1}{q} |\nabla u|^q + \frac{\mu}{q} |\nabla h|^q \\
&\leq (q_+ - 1)\, \frac{\mu\, |\nabla u|^q}{q} + \frac{\mu\, |\nabla h|^q}{q}.
\end{align*}
Since \( u, h \in W^{1,H}(\Omega) \), it follows that \( \mu\, |\nabla u|^q \) and \( \mu\, |\nabla h|^q \) are in \( L^1(\Omega) \). Hence,
\[
\mu\, |\nabla u|^{q-2} (\nabla u \cdot \nabla h) \in L^1(\Omega),
\]
as claimed.  Next, consider the operator \( A_q : W^{1,H}(\Omega) \rightarrow \left( W^{1,H}(\Omega) \right)^* \) defined by
\[
\langle A_q(u), v \rangle = \int_\Omega \mu(x)\, |\nabla u|^{q(x) - 2} \nabla u \cdot \nabla v \, dx.
\]
The energy functional associated with \( A_q \) is the functional  \( {\mathcal I}_q : W^{1,H}(\Omega) \to \mathbb{R} \) defined by
\[
{\mathcal I}_q(u) = \int_\Omega \mu(x)\, \frac{|\nabla u|^{q(x)}}{q(x)}\, dx.
\]
We claim that \( {\mathcal I}_q \) is Gâteaux differentiable and satisfies \( {\mathcal I}_q'(u) = A_q(u) \). Indeed, using the Mean Value Theorem, we obtain
\begin{align*}
\frac{|\nabla u + t \nabla h|^q - |\nabla u|^q}{q t}
&= \frac{1}{q t} \left[ \left( \sum (D_j u + t D_j h)^2 \right)^{q/2} - \left( \sum (D_j u)^2 \right)^{q/2} \right] \\
&= \left[ \sum (D_j u + \theta D_j h)^2 \right]^{\frac{q}{2} - 1} \sum (D_j u + \theta D_j h) D_j h \\
&= |\nabla u + \theta t \nabla h|^{q - 2} (\nabla u + \theta t \nabla h) \cdot \nabla h,
\end{align*}
for some \( \theta = \theta(x, t) \) with \( 0 < \theta(x, t) < 1 \). We estimate:
\begin{align*}
\left| |\nabla u + \theta t \nabla h|^{q - 2} (\nabla u + \theta t \nabla h) \cdot \nabla h \right|
&= |\nabla u + \theta t \nabla h|^{q - 2} \left| (\nabla u + \theta t \nabla h) \cdot \nabla h \right| \\
&\leq |\nabla u + \theta t \nabla h|^{q - 1} |\nabla h|.
\end{align*}
Since \( \nabla u + \theta t \nabla h = (1 - \theta) \nabla u + \theta (\nabla u + t \nabla h) \), we have
\begin{align*}
| \nabla u + \theta t \nabla h |
&= | (1 - \theta) \nabla u + \theta (\nabla u + t \nabla h) | \\
&\leq (1 - \theta) |\nabla u| + \theta (|\nabla u| + |t| |\nabla h|) \\
&= |\nabla u| + \theta |t| |\nabla h| \leq |\nabla u| + |\nabla h|,
\end{align*}
assuming \( |t| \leq 1 \), since we are taking the limit as \( t \to 0 \).\\
Therefore,
\[
\mu\, |\nabla u + \theta t \nabla h|^{q - 1} |\nabla h| \leq \mu\, \left( |\nabla u| + |\nabla h| \right)^{q - 1} |\nabla h|.
\]
As previously shown,
\[
\mu\, \left( |\nabla u| + |\nabla h| \right)^{q - 1} |\nabla h| \leq (q_+ - 1)\, \frac{\mu\, \left( |\nabla u| + |\nabla h| \right)^q}{q} + \frac{\mu\, |\nabla h|^q}{q},
\]
which implies \( \mu\, \left( |\nabla u| + |\nabla h| \right)^{q - 1} |\nabla h| \in L^1(\Omega) \).\\
Putting all the pieces together, we obtain:
\[
\mu\, \frac{|\nabla u + t \nabla h|^q - |\nabla u|^q}{q t} \leq \mu\, \left( |\nabla u| + |\nabla h| \right)^{q - 1} |\nabla h|.
\]
As \( t \to 0 \), we get pointwise convergence:
\[
\mu\, \frac{|\nabla u + t \nabla h|^q - |\nabla u|^q}{q t} \to \mu\, |\nabla u|^{q - 2} (\nabla u \cdot \nabla h).
\]
The Dominated Convergence Theorem then allows us to conclude:
\[
\lim_{t \to 0} \int_\Omega \mu\, \frac{|\nabla u + t \nabla h|^q - |\nabla u|^q}{q t} \, dx = \int_\Omega \mu\, |\nabla u|^{q - 2} (\nabla u \cdot \nabla h)\, dx.
\]
In other words, \( {\mathcal I}_q \) is Gâteaux differentiable and satisfies \( {\mathcal I}_q'(u) = A_q(u) \), which completes the proof of Lemma~\ref{frechet-derivative}.
\end{proof}

\medskip

The following proposition is a direct consequence of Lemma~\ref{frechet-derivative}. It confirms that \( \mathcal{I} \) is Gâteaux differentiable and provides an explicit expression for its derivative.

\begin{proposition}\label{functiondiff}
Under the assumptions stated above, the functional \( \mathcal{I} \) is Gâteaux differentiable, and
\begin{align*}
\langle \mathcal{I}^{\prime}(u), \phi \rangle &= \int_\Omega |\nabla(\varphi - u)|^{p-2} \left( \nabla(\varphi - u) \cdot \nabla \phi \right) \, dx \\
&\quad + \int_\Omega \mu(x)\, |\nabla(\varphi - u)|^{q-2} \left( \nabla(\varphi - u) \cdot \nabla \phi \right) \, dx - \langle f, \phi \rangle,
\end{align*}
for every \( \phi \in W_0^{1,H}(\Omega) \).
\end{proposition}

\medskip
In the sequel, the following variant of Lebesgue's dominated convergence theorem will be used:

\medskip
\begin{lemma}\label{ldc}
Let \( (f_j) \) and \( (g_j) \) be sequences in \( L^1(\Omega) \) such that \( |f_j| \leq g_j \) for each \( j \in \mathbb{N} \).  Suppose that:
\begin{itemize}
    \item \( f_j \to f \) pointwise a.e. in \( \Omega \),
    \item \( g_j \to g \) pointwise a.e. in \( \Omega \), with \( g \in L^1(\Omega) \),
    \item \( \int_\Omega g_j\, dx \to \int_\Omega g\, dx \) as \( j \to \infty \).
\end{itemize}
Then
\[
\int_\Omega f_j\, dx \to \int_\Omega f\, dx \quad \text{as } j \to \infty.
\]
\end{lemma}
\begin{proof}
Let \( (f_j) \) and \( (g_j) \) be sequences in \( L^1(\Omega) \) such that \( |f_j| \leq g_j \) for each \( j \in \mathbb{N} \). Assume further that \( g_j \to g \in L^1(\Omega) \) pointwise a.e. in \( \Omega \), and that
\[
\int_\Omega g_j\, dx \to \int_\Omega g\, dx \quad \text{as } j \to \infty.
\]
By Fatou's Lemma, it follows that \( f \in L^1(\Omega) \). Since \( g_j \to g \) in \( L^1(\Omega) \), the sequence \( (g_j) \) is uniformly integrable. As \( |f_j| \leq g_j \), it follows that \( (f_j) \) is also uniformly integrable. We may now apply Vitali’s Theorem to conclude
\[
\int_\Omega f_j\, dx \to \int_\Omega f\, dx \quad \text{as } j \to \infty.
\]
This completes the proof of Lemma~\ref{ldc}.
\end{proof}

\medskip
The next proposition completes the treatment of the differentiability of the Dirichlet energy functional ${\mathcal I}$.

\medskip
\begin{proposition}\label{functiondifffrechet}
	Under the assumptions stated above, the functional \( \mathcal{I} \) is Fr\'{e}chet differentiable.
\end{proposition}
\begin{proof}
It suffices to show that if $(u_j)\subset W^{1,H}(\Omega)$ and $u_j\rightarrow u$ in $W^{1,H}(\Omega)$, then there exists a subsequence $(u_{j_k})$ of $(u_j)$ such that ${\mathcal I}^{\prime}(u_{j_k})\rightarrow {\mathcal I}^{\prime}(u)$ in the dual topology of $\left(W^{1,H}(\Omega)\right)^{\ast}$. For such $(u_j)$ it holds
\begin{equation}\label{1Htozero}
\rho_{1,H}(u-u_j)\rightarrow 0\,\;\;\text{as}\,\,j\rightarrow\infty;
\end{equation}
accordingly, there is a sequence, still denoted by $(u_j)$, such that $u_j\rightarrow u$ and $\nabla u_j\rightarrow \nabla u$ pointwise a.e. in $\Omega$. From now on, we only consider this sequence. Clearly,
\begin{equation*}
f_j=	\left||\nabla u_j|^{p-2}\nabla u_j-|\nabla u|^{p-2}\nabla u\right|^{\frac{p}{p-1}}\rightarrow 0\,\;\;\text{as}\,j\rightarrow \infty
\end{equation*}
and
\begin{equation*}
	\tilde{f}_j=	\mu\left||\nabla u_j|^{q-2}\nabla u_j-|\nabla u|^{q-2}\nabla u\right|^{\frac{q}{q-1}}\rightarrow 0\,\;\;\text{as}\,j\rightarrow \infty;
\end{equation*}
moreover
\begin{align*}
f_j&\leq \Big( 2^{p_+-2}|\nabla (u_j-u)|^{p-1}+(2^{p_+-2}+1)|\nabla u|^{p-1}\Big)^{\frac{p}{p-1}}\\ 
&\leq2^{\alpha}|\nabla (u-u_j)|^p+2^{\beta-1}(2^{p_+-2}+1)^{\beta}|\nabla u|^p \\
& \leq C(p_+,p_-) \Big(|\nabla (u-u_j)|^p+|\nabla u|^p\Big)=g_j,\quad \text{in}\;\Omega,
\end{align*}
where $\alpha= \beta(p_+-1)-1$, $\beta=\frac{p_+}{p_--1}$ and for a positive constant $C(p_+,p_-)$.  Likewise,
$$	\tilde{f}_j \leq  C(q_+,q_-) \mu \Big(|\nabla (u-u_j)|^q+|\nabla u|^q\Big)=\tilde{g}_j,\quad \text{in}\,\Omega,$$
for a positive constant $C(q_+,q_-)$.\\
It is readily verified that the hypotheses of Lemma \ref{ldc} are satisfied for $f_j,g_j$ and $\tilde{f}_j,\tilde{g}_j$, respectively and that $g_j\rightarrow C(p_+,p_-)|\nabla u|^{p}=g$, $\tilde{g}_j\rightarrow C(q_+,q_-)|\nabla u|^{p}=\tilde{g}$  pointwise a.e. in $\Omega.$ Moreover, the assumption (\ref{1Htozero}) guarantees that $\int\limits_{\Omega}g_jdx\rightarrow \int\limits_{\Omega}gdx$ as $j\rightarrow \infty$ and that $\int\limits_{\Omega}\tilde{g}_jdx\rightarrow \int\limits_{\Omega}\tilde{g}dx$ as $j\rightarrow \infty$, which in turn yield
\begin{equation}\label{fjtozero}
\int\limits_{\Omega}f_jdx\rightarrow 0\,\,\text{and} \,\,\int\limits_{\Omega}\tilde{f}_jdx\rightarrow 0\,\, \text{as}\,\,j\rightarrow \infty.
\end{equation}
Since under the current assumptions, the exponents $p^{\ast}=\frac{p}{p-1}$ and $q^{\ast}=\frac{q}{q-1}$ are bounded in $\Omega$, it follows that (\ref{fjtozero}) is in fact equivalent to 
\begin{equation*}
\Bigg\|\Big||\nabla u_j|^{p-2}\nabla u_j-|\nabla u|^{p-2}\nabla u\Big|\Bigg\|_{p^*}\rightarrow 0\,\,\text{as}\,\,j\rightarrow \infty,
\end{equation*}
and
\begin{equation*}
\Bigg\|\mu \Big||\nabla u_j|^{q-2}\nabla u_j-|\nabla u|^{q-2}\nabla u\Big|\Bigg\|_{q^*}\rightarrow 0\,\,\text{as}\,\,j\rightarrow \infty.
\end{equation*}
It follows now from H\"{o}lder inequality that
\begin{align*}
\|{\mathcal I}^{\prime}(u_j)-{\mathcal I}^{\prime}(u)\|_{\left(W^{1,h}(\Omega)\right)^{\ast}} 
&\leq \sup\limits_{\|w\|_{W^{1,H}} \leq 1}\left|\langle {\mathcal I}^{\prime}(u_j)-{\mathcal I}^{\prime}(u),w\rangle \right| \\ \
&\leq \Big\|\left||\nabla u_j|^{p-2}\nabla u_j-|\nabla u|^{p-2}\nabla u\right|\Big\|_{p^*}\\ 
 &\quad\quad +\Big\|\mu \left||\nabla u_j|^{q-2}\nabla u_j-|\nabla u|^{q-2}\nabla u\right|\Big\|_{q^*},
\end{align*}
i.e., $\|{\mathcal I}^{\prime}(u_j)-{\mathcal I}^{\prime}(u)\|_{\left(W^{1,h}(\Omega)\right)^{\ast}}\rightarrow 0$ as $j\rightarrow \infty.$
\end{proof}

\medskip
\medskip
\section*{Proof of Theorem \ref{core}}

We conclude with the proof of Theorem \ref{core}. By Proposition \ref{functiondiff} and the minimality of \( u \), it follows that for any \( \phi \in C_0^\infty(\Omega) \),
\[
\langle \mathcal{I}'(u), \phi \rangle = \int_\Omega \left( |\nabla(\varphi - u)|^{p-2} \nabla(\varphi - u) + \mu\, |\nabla(\varphi - u)|^{q-2} \nabla(\varphi - u) \right) \cdot \nabla \phi \, dx - \langle f, \phi \rangle.
\]
A standard density argument yields
\[
\int_\Omega \left( |\nabla w|^{p-2} \nabla w + \mu\, |\nabla w|^{q-2} \nabla w \right) \cdot \nabla h \, dx = \langle f, h \rangle,
\]
for all \( h \in W_0^{1,H}(\Omega) \), where \( w = \varphi - u \). Thus, \( w = \varphi - u \) is a weak solution of problem (\ref{DPPP}).\\
The proof of uniqueness relies on the following inequalities (see \cite{Lindqvist}): for all vectors \( A, B \in \mathbb{R}^n \),
\begin{equation}\label{vineq1}
\langle |A|^{r-2} A - |B|^{r-2} B, A - B \rangle \geq \gamma(r) |A - B|^r, \quad \text{for } r \geq 2,
\end{equation}
and
\begin{equation}\label{vineq2}
\langle |A|^{r-2} A - |B|^{r-2} B, A - B \rangle \geq (r - 1) |A - B|^2 \left( 1 + |A|^2 + |B|^2 \right)^{\frac{r - 2}{2}}, \quad \text{for } 1 < r < 2,
\end{equation}
where \( \gamma(r) > 0 \) is a constant depending only on \( r \), and \( \langle \cdot, \cdot \rangle \) denotes the standard inner product in \( \mathbb{R}^n \).\\
Let the sets \( (\Omega_{i,j}) \) be defined as in section \ref{wh-UC}, namely:
\[
\left\{
\begin{aligned}
\Omega_{11} &:=  \{x \in \Omega : 1 < p(x) < 2 \text{ and } 1 < q(x) < 2\},\\
\Omega_{12} &:=  \{x \in \Omega : 1 < p(x) < 2 \text{ and } q(x) \geq 2\},\\
\Omega_{21} &:=  \{x \in \Omega : p(x) \geq 2 \text{ and } 1 < q(x) < 2\},\\
\Omega_{22} &:=  \{x \in \Omega : p(x) \geq 2, \, q(x) \geq 2\}.
\end{aligned}
\right.
\]
Assume \( v \) and \( w \) are two weak solutions of problem (\ref{DPPP}). Then \( v - w = (v - \varphi) - (w - \varphi)\) is in \(W_0^{1,H}(\Omega) \). Thus, for any \( \phi \in C_0^\infty(\Omega) \),
\[
\int_\Omega \left( |\nabla w|^{p-2} \nabla w + \mu\, |\nabla w|^{q-2} \nabla w \right) \cdot \nabla \phi \, dx = \langle f, \phi \rangle,
\]
and
\[
\int_\Omega \left( |\nabla v|^{p-2} \nabla v + \mu\, |\nabla v|^{q-2} \nabla v \right) \cdot \nabla \phi \, dx = \langle f, \phi \rangle.
\]
Taking \( \phi = v - w \in W_0^{1,H}(\Omega) \), we obtain:
\begin{align*}
0 &= \sum_{i,j} \int_{\Omega_{i,j}} \left( |\nabla w|^{p-2} \nabla w - |\nabla v|^{p-2} \nabla v \right) \cdot \nabla (v - w) \, dx \\
&\quad + \sum_{i,j} \int_{\Omega_{i,j}} \mu(x) \left( |\nabla w|^{q-2} \nabla w - |\nabla v|^{q-2} \nabla v \right) \cdot \nabla (v - w) \, dx \\
&\geq (p_- - 1) \int_{\Omega_{11}} |\nabla (v - w)|^2 \left( 1 + |\nabla v|^2 + |\nabla w|^2 \right)^{\frac{p - 2}{2}} dx \\
&\quad + (q_- - 1) \int_{\Omega_{11}} \mu\, |\nabla (v - w)|^2 \left( 1 + |\nabla v|^2 + |\nabla w|^2 \right)^{\frac{q - 2}{2}} dx \\
&\quad + (p_- - 1) \int_{\Omega_{12}} |\nabla (v - w)|^2 \left( 1 + |\nabla v|^2 + |\nabla w|^2 \right)^{\frac{p - 2}{2}} dx \\
&\quad + \int_{\Omega_{12}} \mu\, \gamma(q)\, |\nabla (v - w)|^q dx \\
&\quad + \int_{\Omega_{21}} \gamma(p)\, |\nabla (v - w)|^p dx \\
&\quad + (q_- - 1) \int_{\Omega_{21}} |\nabla (v - w)|^2 \left( 1 + |\nabla v|^2 + |\nabla w|^2 \right)^{\frac{q - 2}{2}} dx \\
&\quad + \int_{\Omega_{22}} \gamma(p)\, |\nabla (v - w)|^p dx + \int_{\Omega_{22}} \mu\, \gamma(q)\, |\nabla (v - w)|^q dx.
\end{align*}
Hence, all integrals are nonnegative and their sum is zero, which implies \( \nabla(v - w) = 0 \). Therefore, \( v = w \), and uniqueness follows.\\
This completes the proof of Theorem~\ref{core}.

\medskip
\medskip
\section*{Acknowledgment}
The authors gratefully acknowledge the support and encouragement of their respective institutions during the preparation of this work.

\medskip
\medskip
\noindent{\bf Declaration of Generative AI and AI-Assisted Technologies in the Writing Process}

During the preparation of this work, the authors used ChatGPT (OpenAI) to assist with language polishing and editorial refinement. All mathematical content, results, and conclusions were developed independently by the authors, who take full responsibility for the content of this publication.

\end{document}